\documentclass[runningheads]{llncs}

\usepackage{amsmath}%
\usepackage{amssymb}%
\usepackage{amsfonts}%
\usepackage{stmaryrd}
\usepackage{mathrsfs}%
\usepackage{textcomp}%
\usepackage{manyfoot}%
\usepackage{booktabs}%
\usepackage{orcidlink}
\usepackage{comment}
\usepackage[inline]{enumitem} 

\usepackage{mathbbol}
\usepackage{proof}

\renewcommand{\varphi}{\phi}
\renewcommand{\to}{\supset}
\newcommand{\fillBox}{\hfill$\Box$}

\newcommand{\seq}{\Rightarrow }
\newcommand{\rn}[1]{\mathsf{#1}}
\newcommand{\ern}[1]{\rn{#1}_\mathsf{E}}
\newcommand{\irn}[1]{\rn{#1}_\mathsf{I}}

\newcommand{\base}[1]{\mathscr{#1}}
\newcommand{\baseB}{\base{B}}
\newcommand{\baseC}{\base{C}}
\newcommand{\baseD}{\base{D}}
\newcommand{\baseE}{\base{E}}

\newcommand{\baseX}{\base{X}}

\newcommand{\At}{\mathbb{A}}
\newcommand{\Worlds}{\mathbb{W}}
\newcommand{\Basic}{\mathbb{B}}

\newcommand{\baseGeq}{\supseteq}

\newcommand{\deriveBaseIML}[1]{\vdash_{\!\!#1}^{ \!\! \gamma}}

\newcommand{\rel}{\mathfrak{R}}
\newcommand{\suppIML}[1]{\Vdash_{ \!\!#1 }^{ \!\! \gamma}}
\newcommand{\suppM}[2]{\Vdash_{ \!\!#1 }^{ \!\!#2 }}

\newcommand{\flatIML}[1]{({#1})^{\flat}}
\newcommand{\deflatIML}[1]{({#1})^{\natural}}

\newcommand{\deflatILL}[1]{{#1}^{\natural}}

\DeclareMathSymbol{\msetsum}{\mathrel}{bbold}{\lq\,}

\makeatletter
\def\labelandtag#1#2{\begingroup
   \def\@currentlabel{#2}%
   \phantomsection\label{#1}\endgroup
}
\makeatother

\def\descriptionlabel#1{\hspace\labelsep \upshape #1}
\makeatletter
\let\orgdescriptionlabel\descriptionlabel
\renewcommand*{\descriptionlabel}[1]{%
  \let\orglabel\label
  \let\label\@gobble
  \phantomsection
  \edef\@currentlabel{#1}%
  \let\label\orglabel
  \orgdescriptionlabel{(#1)}%
}
\makeatother

\begin{document}

\title{Base-extension Semantics for \\ Intuitionistic Modal Logics}
\titlerunning{B-eS for Intuitionistic Modal Logics}
%
\author{Yll Buzoku\inst{1}\orcidID{0009-0006-9478-6009} 
\and  
David J. Pym\inst{1,2}\orcidID{0000-0002-6504-5838}}
\authorrunning{Yll Buzoku \and David Pym}
%
\institute{University College London, London WC1E 6BT, UK \and
Institute of Philosophy, University of London,  London WC1H 0AR, UK
\email{\{y.buzoku@ucl.ac.uk,d.pym\}@ucl.ac.uk}}
\maketitle     

\begin{abstract}
The proof theory and semantics of intuitionistic modal logics have been studied by Simpson in terms of Prawitz-style labelled natural deduction systems and Kripke models. An alternative to model-theoretic semantics is provided by proof-theoretic semantics, which is a logical realization of inferentialism, in which the meaning of constructs is understood through their use. The key idea in proof-theoretic semantics is that of a base of atomic rules, all of which refer only to propositional atoms and involve no logical connectives. A specific form of proof-theoretic semantics, known as base-extension semantics (B-eS), is concerned with the validity of formulae and provides a direct counterpart to Kripke models that is grounded in the provability of atomic formulae in a base. We establish, systematically, B-eS for Simpson's intuitionistic modal logics and, also systematically, obtain soundness and completeness theorems with respect to Simpson's natural deduction systems. 

\keywords{inferentialism \and 
proof-theoretic semantics \and 
base-extension semantics \and 
intuitionistic modal logics \and
labelled natural deduction}
\end{abstract}

\section{Introduction} \label{sec:introduction}

Simpson \cite{simpson1994Thesis} introduces a family of natural deduction systems for a range of intuitionistic modal logics (iMLs). These natural deduction systems are shown to be both sound and complete with respect to meta-logical derivations in the intuitionistic first-order meta-theory that is assumed. This result is shown by defining a bidirectional translation of modal formulae to intuitionistic first-order formulae and showing that theorems of the natural deduction systems, under the translation, correspond to a particular type of (intuitionistic) first-order derivation. Simpson also defines a Kripke semantics for his intuitionistic modal logics based on the interpretation of first-order expressions in the Kripke semantics of first-order intuitionistic logic. The resulting modal models are then used to establish the soundness and completeness of his natural deduction systems. 

Proof-theoretic semantics is an alternative conception of a semantic theory that is grounded in the philosophical position known as inferentialism, in which the meanings of language constructs are considered to be determined by their use. In proof-theoretic semantics, which may be seen as a logical realization of the inferentialist position, meaning is given 
to a logical language through proof, rather than truth. This idea has deep roots in the works of Wittgenstein, with his principle that the meaning of a word should be reflected in its  use~\cite{Wittgenstein1953-WITPI-4}, but also more recently in works such as those of Dummett~\cite{Dummett_LogicalBasis_1991} and Brandom~\cite{brandom_MakingItExplicit,brandom_ArticulatingReasons,brandom_ReasonsForLogic}. 

Proof-theoretic semantics can usefully be seen as having two main branches. One, proof-theoretic validity (P-tV) is concerned --- in a sense articulated by Prawitz \cite{Prawitz1971ideas}, Dummett \cite{Dummett_LogicalBasis_1991}, and Schroeder-Heister \cite{Schroeder-Heister1991-SCHUPS,Schroeder2007modelvsproof,Piecha2015failure,Piecha_CompletenessInPtS_2016} --- with what constitutes a valid proof. The other, base-extension semantics (B-eS) is concerned with what constitutes a valid formula and is the focus of this paper: we provide a uniform B-eS for the iMLs studied by Simpson~\cite{simpson1994Thesis}.  

Sandqvist \cite{Tor2015} obtains a sound and complete B-eS for intuitionistic propositional logic (IPL). This is achieved by defining the concept of an atomic rule as an inference figure such as $\mathcal{R}$ in Figure~\ref{fig:B-eS-IPL} (which we may write linearly as $((P_1 \Rightarrow q_1) , \ldots , (P_n \Rightarrow q_n)) \Rightarrow r)$), and then defining a relation of atomic derivability $\vdash_{\baseB}$ over sets of atomic rules $\baseB$, between sets of atoms and atoms. Crucially, the rules make mention only of atoms and the relation of atomic derivability, as the name suggests, is thus capable only of observing judgements between sets of atoms and individual atoms. One can imagine the $\baseB$ as being a natural deduction like construct whose elements are not schemas, but, in fact, instances of rules. Thus, in the same way that $\rm NJ$ can be viewed as generating the consequence relation $\vdash_{\rm NJ}$, the relation $\vdash_\baseB$ can be understood as the consequence relation generated by $\baseB$. Sandqvist then conservatively extends this relation to a relation of support $\suppM{\baseB}{}$ over a base $\baseB$ which relates sets of IPL formulae to an IPL formula. 
\begin{figure}[ht]
\hrule 
\vspace{2mm}
\[
\frac{\begin{array}{ccc}
        [P_1] &        & [P_n] \\
        q_1   & \ldots & q_n  
      \end{array}
}{r} \, \mathcal{R}
\qquad 
{\small\begin{array}{rl}
\mbox{(Ref)} & \mbox{$P , p \vdash_\baseB p$} \\ 
(\mbox{App}_\mathcal{R}) & \mbox{if $((P_1 \Rightarrow q_1) , \ldots , (P_n \Rightarrow q_n)) \Rightarrow r)$ and,} \\
    & \mbox{for all $i \in [1,n]$, $P , P_i \vdash_\baseB q_i$, 
    then $P \vdash_\baseB r$} 
\end{array}}
\]
\[{\small
\begin{array}{rl@{\quad}rl} 
\mbox{(At)} & \mbox{for atomic $p$, $\Vdash_\baseB p$ iff $\vdash_\baseB p$} & 
    (\vee) & \mbox{$\Vdash_\baseB \phi \!\vee\! \psi$ iff, for every atomic $p$ and every } \\ 
    & & & \mbox{$\baseC \!\supseteq\! \baseB$, if $\phi \Vdash_\baseC p$ and $\psi \Vdash_\baseC p$, then $\Vdash_\baseC p$}\\
(\supset) & \mbox{$\Vdash_\baseB \phi \supset \psi$ iff $\phi \Vdash_\baseB \psi$} & (\bot) & 
    \mbox{$\Vdash_\baseB \bot$ iff, for all atomic $p$, $\Vdash_\baseB p$} \\
(\wedge) & \mbox{$\Vdash_\baseB \phi \wedge \psi$ iff $\Vdash_\baseB \phi$ and 
    $\Vdash_\baseB \psi$} & \mbox{(Inf)} & \mbox{for $\Theta \neq \emptyset$, 
        $\Theta \Vdash_\baseB \phi$ iff, for every $\baseC \supseteq \baseB$,}  \\ 
    & & & \mbox{if $\Vdash_\baseC \theta$, for every $\theta \in \Theta$, then $\Vdash_\baseC \phi$}     
\end{array}}
\] \vspace{-4mm}
\caption{Sandqvist's B-eS for Intuitionistic Propositional Logic}
\label{fig:B-eS-IPL} 
\vspace{2mm}
\hrule
\end{figure}  

The need to consider base extensions, of the form $\baseC \supseteq \baseB$, in entailments defined by $\Vdash_{\baseB}$ is analogous to the need to consider accessible worlds in Kripke models of intuitionistic implication. As Prawitz \cite{Prawitz1971ideas} 
explains, we wish the semantics to yield a construction of an implication $\phi \supset \psi$, say, as a construction of $\psi$ from $\phi$ that, together with a construction of $\phi$, yields a construction of $\psi$. For a fixed base $\baseB$, the condition as formulated above would be vacuously satisfied in the absence of a construction of $\phi$ relative to $\baseB$. This would be a rather weak condition and, unsurprisingly, the semantics requires that all base-extensions $\baseC \supseteq \baseB$ be used. See also \cite{GP2025} to explore further some related issues.  

It is a support relation of this form that provides our main semantic tool in giving a uniform B-eS for iMLs. 

The semantics as defined in Figure~\ref{fig:B-eS-IPL} is sound and complete for IPL. In this framework, an IPL sequent $(\Gamma:\varphi)$ is deemed valid iff for all bases $\baseB$, we have that $\Gamma\suppM{\baseB}{}\varphi$ (where $\Gamma$ is a set of IPL formulae and $\varphi$ is an IPL formula). This is somewhat analogous to what one might expect in a Kripke-like semantics, where one would say something like $\varphi$ is valid iff for all models $\mathcal{M}$ we have that $\mathcal{M}\vDash\varphi$. It is crucial, though, that bases are \emph{not} the same as models and that this difference is easily observed in the fact that the meanings of the connectives --- that is, the semantic clauses --- are very different in some cases from the familiar Kripke semantics for IPL. For example, 
the form of the clause for disjunction in Figure~\ref{fig:B-eS-IPL} can be understood either as following the second-order definition of $\vee$ or, perhaps more proof-theoretically, as following the form of the $\vee$-elimination rule in NJ (see, e.g., \cite{Prawitz_NatDed_2006,Pym-Ritter-Robinson_CategoricalP-tS_2025,GP2025}) --- cf. also Beth's semantics; for example, \cite{LS86} provides an illuminating discussion. That said, 
the treatment of $\bot$ follows, for example, Dummett \cite{Dummett_LogicalBasis_1991}. 
Sandqvist proves, in \cite{Tor2015}, soundness and completeness of the B-eS relative to NJ \cite{Prawitz_NatDed_2006}. We sketch the arguments below. 

\begin{theorem}[IPL Soundness] \label{thm:ipl-sound}
If $\Gamma \vdash_{\rm NJ} \phi$, then $\Gamma \Vdash \phi$. \fillBox
\end{theorem}
In this case, by the inductive definition of $\vdash_{\rm NJ}$, it is sufficient to consider simply every rule of $\rm NJ$ as a series of meta-logical expressions with us assuming the hypothesis of every rule is valid and showing that the conclusion of every rule is also valid. For example, for $\irn \to$, we show that if $\Gamma,\varphi\Vdash_{\emptyset}\psi$,  then $\Gamma\Vdash_{\emptyset}\varphi\to\psi$ (note, using the base $\emptyset$ is equivalent to the validity condition mentioned above).

\begin{theorem}[IPL Completeness] \label{thm:ipl-complete} If $\Gamma \Vdash \phi$, then $\Gamma  \vdash_{\rm NJ} \phi$. \fillBox
\end{theorem}
The proof of completeness is more involved, but amounts to showing that, for any valid sequent, we can always construct a base $\base{N}$ whose rules simulate $\rm NJ$. Since $\base{N}$ does not contain schemas, every ${\rm NJ}$ rule must be simulated for every subformula of the valid sequent. Since the valid sequent is indeed valid, we can therefore construct an $\rm NJ$ proof of the sequent using the rules of $\base{N}$, and then translate it faithfully to give us an $\rm NJ$ proof of the valid sequent.

For B-eS for iMLs, we take Simpson \cite{simpson1994Thesis} and Sandqvist \cite{Tor2015} as starting points and combine them to give B-eS for iMLs with natural deduction systems defined by Simpson in~\cite{simpson1994Thesis}. We do this in a very similar fashion to that described for IPL \cite{Tor2015}. Crucially, however, more organizational structure within the formulation of bases is necessary.

As mentioned above, Simpson \cite{simpson1994Thesis} shows that there is a translation, the standard translation, between intuitionistic first-order formulae and the modal formulae that he considers for his proof systems. The modal formulae would better be described as being `labelled' modal formulae, whose intuitive reading is as being the `world' at which that formula holds. Similarly, Simpson allows additionally for a special type of \emph{non-logical} object to be present in the hypothesis of any rules in his natural deduction systems; that of relations between the worlds. 

These objects are, in the inferentialist spirit of bases, inherently non-logical --- they simply describe which labels are related to one other and do not involve logical constants. Accordingly, we consider these objects alongside atomic propositions (to which we now also attach a label) to be `first-class citizens' in bases. Consequently, our notion of atomic rule can contain both world relations and labelled atoms. However, in keeping with the spirit of the natural deduction systems of Simpson, we do not generally allow for a rule to conclude such relations. That is to say, whereas it might be desirable to be able to directly express rules of the form $\mathcal{R}$, we can, and do, express them in form $\mathcal{R}'$ 
\[
 \infer[\mathcal{R}]{xRz}{xRy & yRz}  
\qquad
 \infer[\mathcal{R}']{p^v}{xRy & yRz & \deduce{p^v}{[xRz]}} 
\]
where $p^v$ is an atom. Doing so allows us to build bases whose proof structures mimic those of the natural deduction systems of Simpson. Thus, soundness and, especially, completeness are proven exactly as in \cite{Tor2015}; however, extra care is taken with respect to the modal connectives. 

Building on all this, we give a brief summary of the organization of this paper. Section~\ref{sec:IMLs} introduces natural deduction systems for intuitionistic modal logics in the sense of Simpson \cite{simpson1994Thesis}. Section~\ref{sec:BDS} introduces the notion of derivability in a base, extending the work of Sandqvist \cite{Tor2015} to include labelled atoms and relations on labels. This notion of derivability is extended to incorporate labelled formulae and to give a conception of a semantic theory through relational, or `frame', properties. From this grounding, a B-eS in the sense of Sandqvist \cite{Tor2015} is given for the 
family of modal logics described by the relational properties. Sections \ref{sec:soundness} and \ref{sec:completeness} then establish that this semantics is sound and complete with respect to the natural deduction systems given by Simpson for their respective logics --- some of the details of the proofs are deferred to \cite{full}. We conclude, in Section~\ref{sec:conclusion}, with a summary of our contribution and a brief discussion of directions for further research. 

In recent work by Gheorghiu, Gu, and Pym \cite{GGP2024} on inferentialist semantics for substructural logics --- and also work by Eckhardt and Pym \cite{EP1,EP2} on classical modal logics --- it has been shown that B-eS has significant potential as a modelling tool. As well as being a fundamental contribution in logic, demonstrating the scope of proof-theoretic semantics, the work presented herein adds to the collection of logics that are available to support that line of enquiry. 

\section{Intuitionistic Modal Logics} \label{sec:IMLs} 

We begin by fixing some definitions that are required throughout the paper. Discussions of these definitions can be found in Simpson's comprehensive `The Proof Theory and Semantics of Intuitionistic Modal Logics' \cite{simpson1994Thesis}.  

Fix a countable set of variables $\Worlds$ which we call labels, and propositional atoms $\At$.

\begin{definition}[Modal formulae]
Propositional modal formulae are defined by the 
grammar:
$
\quad \varphi,\psi ::= p\in \At\,|\,\varphi\land\psi\,|\,\varphi\lor\psi\,|\,\varphi\to\psi\,|\,\Box\varphi\,|\,\Diamond\varphi\,|\,\bot\,|\, \top 
$
\fillBox
\end{definition}

We enrich formulae to labelled formulae. The intuitive reading of such objects is that they express the `locale' at which the formula holds true. 

\begin{definition}[Labelled formula]
A labelled formula is an ordered pair $\langle x,\varphi\rangle$, which we write as $\varphi^x$, where $x\in \Worlds$ and $\varphi$ is a propositional modal formula. \fillBox    
\end{definition}

\begin{definition}[Sequent]
An intuitionistic modal sequent is an ordered pair $\langle\Gamma,\varphi^x\rangle$ which we write as  $(\Gamma:\varphi^x)$, where $\Gamma$ is a (finite) set of labelled formulae and $\varphi^x$ is a labelled formula. \fillBox
\end{definition}


We introduce relational assumptions, treated similarly to propositional atoms.

\begin{definition}[Relational assumption]
A relational assumption is an ordered pair $\langle x,y \rangle$ where $x,y\in \Worlds$, though we will write this as $xRy$. \fillBox
\end{definition}


Relational assumptions act like edges of a directed graph. This intuition will be useful for understanding the proof theory of the intuitionistic modal logics. To make this precise, recall the definition of a graph.

\begin{definition}[Graph]
A graph is a pair $\mathcal{G} = (X,\rel)$ where $X\subseteq\Worlds$ is a non-empty set of labels and $\rel$ is a binary relation on $X$. We write $x\rel y$ to mean that $xRy\in\rel$. \fillBox
\end{definition}

Given a graph $\mathcal{G}=(X,\rel)$, we can require that the relations be subject to 
different global properties (i.e., properties on all $x\in X$) such as those in Figure~\ref{fig:frameConditions}. The properties in Figure~\ref{fig:frameConditions} are conventionally called \emph{frame conditions} in the context of modal logic and we adopt this terminology herein. We denote a (possibly empty) subset of these frame conditions and represent this by a letter $\gamma$, which should be understood as $\gamma\subseteq\{\gamma_D,\gamma_T,\gamma_B,\gamma_4,\gamma_5,\gamma_2\}$. By fixing such a $\gamma$, we 
fix a particular intuitionistic modal logic, with the empty subset $\gamma=\emptyset$ denoting the modal logic $\rm iK$. 

\begin{figure}[ht]
    \centering
    \hrule
\[{\small
\begin{array}{l@{\quad}r@{\quad}l@{\quad}l}
\mbox{Axiom Schema} & \mbox{Label} & \mbox{Name} & \mbox{Relational property} \\
 & & & \\ 
\mbox{$\Diamond \top$} &\mbox{$\gamma_D$} & \mbox{Seriality} & \mbox{$\forall x.\,\exists y.\,xRy$} \\
\mbox{$\Box\varphi\to\varphi$} &\mbox{$\gamma_T$} & \mbox{Reflexivity} & \mbox{$\forall x.\,xRx$} \\
\mbox{$\varphi\to\Box\Diamond\varphi$} &\mbox{$\gamma_B$} & \mbox{Symmetry} & \mbox{$\forall x.\,\forall y.\,xRy \Rightarrow yRx$} \\
\mbox{$\Box\varphi\to\Box\Box\varphi$} &\mbox{$\gamma_4$} & \mbox{Transitivity} & \mbox{$\forall x.\,\forall y.\,\forall z.\,xRy \;\&\; yRz \Rightarrow xRz$} \\
\mbox{$\Diamond\varphi\to\Box\Diamond\varphi$} &\mbox{$\gamma_5$} & \mbox{Euclidean} & \mbox{$\forall x.\,\forall y.\,\forall z.\,xRy \;\&\; xRz \Rightarrow yRz$} \\
\mbox{$\Diamond\Box\varphi\to\Box\Diamond\varphi$} &\mbox{$\gamma_2$} & \mbox{Directed} & \mbox{$\forall x.\,\forall y.\,\forall z.\,xRy\; \&\; yRz \Rightarrow \exists w.\,yRw \; \&\; zRw$} 
\end{array}}
\] \vspace{-6mm}
    \caption{Frame conditions} \vspace{2mm}
    \label{fig:frameConditions}
    \hrule
\end{figure}

\begin{definition}[Intuitionistic modal derivability]\label{def:modalDerivability}
    Given a set of frame conditions $\gamma$, a sequent $(\Gamma:\varphi^x)$, and a graph $\mathcal{G}=(X,\rel)$ whose vertex set $X$ contains at least every label of the elements of $\Gamma$ and $x$, whose edge-set $\rel$ satisfies the conditions $\gamma$, then the relation of derivability for the specific intuitionistic modal logic satisfying the modal axioms corresponding to the frame conditions $\gamma$ is defined by the schemas of Figure~\ref{fig:NatDedIML} and the schemas corresponding to the frame conditions of Figure~\ref{fig:NatDedProperties}. The resulting consequence relation is written $\Gamma\vdash^\gamma_\mathcal{G}\varphi^x$. \fillBox
\end{definition}

The schemas of Figure~\ref{fig:NatDedIML} form the base natural deduction system $\rm N_{\Box\Diamond}$, which corresponds to a natural deduction system for the logic $\rm iK$. We extend this natural deduction system to other intuitionistic modal logics by adding to $\rm N_{\Box\Diamond}$ additional inference figures from Figure~\ref{fig:NatDedProperties}, which represent the necessary relational properties of those logics. That this is justified is shown by Simpson in~\cite{simpson1994Thesis}. We write $\rm N_{\Box\Diamond}(\gamma)$ to denote the natural deduction system $N_{\Box\Diamond}$ extended by the rules that correspond to the frame conditions $\gamma$. Thus, the consequence relation $\Gamma\vdash_{\mathcal{G}}^\gamma\varphi^x$ should be read as saying that there is a derivation in $N_{\Box\Diamond}(\gamma)$ of $\varphi^x$ from open assumptions $\Gamma$ and $\rel$, where $\mathcal{G}=(X,\rel)$, recalling that $X=\{x\,|\,\psi^x\in\Gamma\cup\{\varphi^x\}\}$. Of interest is the trivial graph: $\tau = (x,\emptyset)$. 

The soundness and completeness of the natural deduction systems $N_{\Box\Diamond}(\gamma)$ with respect to their intuitionistic modal logics is proven in~\cite{simpson1994Thesis}. 

\begin{definition}[Theorem of $\rm N_{\Box\Diamond}(\gamma)$]
    A labelled formula $\varphi^x$ is called a theorem of $\rm N_{\Box\Diamond}(\gamma)$ if $\vdash^\gamma_\tau\varphi^x$ holds. We write this as $\vdash^\gamma\varphi^x$. \fillBox
\end{definition}


\begin{figure}[t]
    \hrule \vspace{-1mm}
    \[{\small 
    \begin{array}{c@{\qquad}c} 
        \infer[\irn \top]{~~\top^x~~}{} & \infer[\ern \bot]{\varphi^y}{\bot^x} \\[3mm]
        \infer[\irn \to]{(\varphi \to \psi)^x}{\deduce{\psi^x}{[\varphi^x]}} & 
        \infer[\ern \to]{\psi^x}{(\varphi \to \psi)^x & \varphi^x} \\[3mm]
        \infer[\irn \land]{(\varphi\land\psi)^x}{\varphi^x & \psi^x} & 
        \infer[\ern{\land_1}]{\varphi^x}{(\varphi \land \psi)^x} \quad \infer[\ern{ \land_2}]{\psi^x}{(\varphi \land \psi)^x} \\[3mm] 
        \infer[\irn{\lor_1}]{(\varphi\lor\psi)^x}{\varphi^x}\quad \infer[\irn{\lor_2}]{(\varphi\lor\psi)^x}{\psi^x} & 
        \infer[\ern \lor]{\chi^y}{(\varphi\lor\psi)^x & \deduce{\chi^y}{[\varphi^x]} & \deduce{\chi^y}{[\psi^x]}}\\[3mm]
        \infer[\irn \Box^*]{(\Box\varphi)^x}{\deduce{\varphi^y}{[xRy]}}&
        \infer[\ern \Box]{\varphi^y}{(\Box\varphi)^x & xRy}\\[3mm]
        \infer[\irn \Diamond]{(\Diamond\varphi)^x}{\varphi^y & xRy}&
        \infer[\ern \Diamond^{**}]{\psi^z}{(\Diamond\varphi)^x & \deduce{\psi^z}{[\varphi^y]\,[xRy]}}\\[3mm]
        *\text{ The label $y$ is different to $x$ and} & **\text{ The label $y$ is different to $x$ and $z$}\\
        \text{the labels of any open assumptions.} & \text{ and the labels of any open assumptions.}
    \end{array}}
    \] \vspace{-5mm}
     \caption{The natural deduction system $\rm N_{\Box\Diamond}$ for intuitionistic modal logic $\rm iK$} \vspace{2mm}
    \hrule
    \label{fig:NatDedIML}
\end{figure}

\begin{figure}[ht]
    \hrule
    \[{\small
    \begin{array}{c@{\qquad}c}
        \infer[(R_D)^*]{\varphi^z}{\deduce{\varphi^z}{[xRy]}} & 
        \infer[(R_T)]{\varphi^y}{\deduce{\varphi^y}{[xRx]}} \\[3mm]
        \infer[(R_B)]{\varphi^z}{xRy & \deduce{\varphi^z}{[yRx]}} & 
        \infer[(R_4)]{\varphi^w}{xRy & yRz & \deduce{\varphi^w}{[xRz]}} \\[3mm] 
        \infer[(R_5)]{\varphi^w}{xRy & xRz & \deduce{\varphi^w}{[yRz]}} &
        \infer[(R_2)^{**}]{\varphi^v}{xRy & xRz & \deduce{\varphi^v}{[yRw]\,[zRw]}} \\[3mm] 
        *\text{ The label $y$ is different to $x$ and} & **\text{ The label $w$ is different to $v,x,y,z$}\\
        \text{the labels of any open assumptions.} & \text{ and the labels of any open assumptions.}
    \end{array}}
    \] \vspace{-4mm}
    \caption{Rules which extend $\rm N_{\Box\Diamond}$ that express the properties of the relational assumptions} \vspace{2mm}
    \hrule
    \label{fig:NatDedProperties}
\end{figure}

\section{Derivability in a Base and Base-extension Semantics} \label{sec:BDS}

Following the approach of Sandqvist \cite{Tor2015}, we start by defining a notion of atomic derivability, which is the basis of the base-extension semantics.

\begin{definition}[Basic sentence]
    A basic sentence is either a labelled propositional letter or a relational assumption. 
    $\Basic$ denotes the set of all basic sentences. \fillBox
\end{definition}

For the remainder of this paper, we adopt the convention that lower case Latin letters refer to basic sentences, except for the variable letters $v,w,x,y,z,a,b$. Lower case Greek letters, except for $\gamma$, will refer to intuitionistic modal formulae and uppercase Latin and Greek letters, except for the relation symbol $R$, will refer to finite sets thereof. A basic sentence written without a superscript is taken to mean \emph{either} a labelled propositional letter \emph{or} a relational assumption. We also henceforth use the word atom to mean labelled \emph{atom}. 

\begin{definition}[Basic sequent]
    A basic sequent is an ordered pair $\langle P,p\rangle$ where $P$ is a (finite) set of basic sentences and $p$ is an 
    atom if $P\neq\emptyset$. Else $p$ is a basic sentence. We write this pair as $P\seq p$ and omit the $P$ if it is empty. \fillBox
\end{definition}

\begin{definition}[Basic rule]
    A basic rule is an ordered pair $\langle\mathbf{Q},r\rangle$ where $\mathbf{Q}$ is a (finite) set of basic sequents and $r$ is a labelled atom if $\mathbf{Q}\neq\emptyset$. Else, $r$ is a basic sentence. We write this as $(P_1\seq p_1, \dots, P_n\seq p_n)\seq r$ where each $P_i\seq p_i\in \mathbf{Q}$ and omit the $(P_1\seq p_1, \dots, P_n\seq p_n)$ if $\mathbf{Q}$ is empty. \fillBox
\end{definition}

\begin{definition}[Base]
    A base is a set of basic rules. \fillBox
\end{definition}

\begin{definition}[Basic derivability relation]
    Given a set of frame conditions $\gamma$ and a base $\baseB$, we can inductively define a relation of derivability on basic sequents as follows:
    \[
    \begin{array}{r@{\quad}l}
    \mbox{\rm (Ref)} & \mbox{$S, p \deriveBaseIML{\baseB}p$} \\[1mm]
    \mbox{\rm (App)} & \mbox{If $((P_1\seq p_1, \dots, P_n\seq p_n)\seq r)\in \baseB$ and} \\
    & \mbox{\quad$S,P_i\deriveBaseIML{\baseB}p_i$ for each $i$, then $S\deriveBaseIML{\baseB}r$}\\[1mm]
    \end{array} 
    \]
    \[
    \begin{array}{r@{\quad}l}
    \mbox{\rm (D)} & \mbox{If $\gamma_D\in\gamma$ and there exists a y such that  $S,xRy\deriveBaseIML{\baseB}p^z$, then $S\deriveBaseIML{\baseB}p^z$} \\[1mm]
    \mbox{\rm (T)} & \mbox{If $\gamma_T\in\gamma$ and $S,xRx\deriveBaseIML{\baseB}p^y$, then $S\deriveBaseIML{\baseB}p^y$}\\[1mm]
    \mbox{\rm (B)} & \mbox{If $\gamma_B\in\gamma$, $S\deriveBaseIML{\baseB}xRy$ and $S,yRx\deriveBaseIML{\baseB}p^z$, then $S\deriveBaseIML{\baseB}p^z$}\\[1mm]
     \end{array} 
    \]
    \[
    \begin{array}{r@{\quad}l}
    \mbox{\rm (4)} & \mbox{If $\gamma_4\in\gamma$, $S\deriveBaseIML{\baseB}xRy$, $S\deriveBaseIML{\baseB}yRz$, and $S,xRz\deriveBaseIML{\baseB}p^w$, then $S\deriveBaseIML{\baseB}p^w$}\\[1mm]
    \mbox{\rm (5)} & \mbox{If $\gamma_5\in\gamma$, $S\deriveBaseIML{\baseB}xRy$, $S\deriveBaseIML{\baseB}xRz$, and $S,yRz\deriveBaseIML{\baseB}p^w$, then $S\deriveBaseIML{\baseB}p^w$}\\[1mm]
    \mbox{\rm (2)} & \mbox{If $\gamma_2\in\gamma$, $S\deriveBaseIML{\baseB}xRy$, $S\deriveBaseIML{\baseB}xRz$, and there exists a $w$ such that}\\
    & \quad\mbox{$S,yRw,zRw\deriveBaseIML{\baseB}p^v$, then $S\deriveBaseIML{\baseB}p^v$}. 
    \end{array} 
    \] 
The last six cases we call the \emph{modal cases} of the derivability relation.
In the case of $\rm (D)$, the label $y$ cannot appear as the label of any $s\in S$ nor be equal to $x$ or $z$. In the case of $(2)$, the label $w$ cannot appear as the label of any $s\in S$ nor be equal to $x,y,z,v$.   \fillBox
\end{definition}

\begin{lemma}\label{lem:basicWeakening}
    If $U\deriveBaseIML{\baseB}u$, then $T,U\deriveBaseIML{\baseB}u$. 
\end{lemma}
\begin{proof}
    We start noting that, in the modal cases, we are immediately done. So now consider when $U\deriveBaseIML{\baseB}u$ holds by (Ref), in which case $u \in U$. Therefore, $u\in T,U$ and thus $T,U\deriveBaseIML{\baseB}u$ holds again by (Ref). Otherwise, consider when $U\deriveBaseIML{\baseB}u^x$ holds by (App). In this case, there must exist a rule $((P_1\seq p_1, \dots, P_n\seq p_n)\seq u)\in \baseB$ such that $U,P_i\deriveBaseIML{\baseB}p_i$ holds for each $i$. By the inductive hypothesis, we therefore obtain that $T,U,P_i\deriveBaseIML{\baseB}p_i$ for each $i$, from which we conclude by (App) that $T,U\deriveBaseIML{\baseB}u$. \fillBox
\end{proof}

\begin{lemma}\label{lem:basicInf}
    $T\deriveBaseIML{\baseB}u$ if and only if, for all $\baseC\baseGeq\baseB$ such that if, for all $t\in T$, we have $\deriveBaseIML{\baseC}t$, then $\deriveBaseIML{\baseC}u$.
\end{lemma}

\begin{proof}
Going left to right, we suppose the hypothesis and consider how $T\deriveBaseIML{\baseB}u$ holds. We consider just the non-modal cases as the modal cases are immediate. If it holds by (Ref), then $u\in T$ and thus we are immediately done. Else, it must hold by (App) in which case we have that there exists a rule $((P_1\seq p_1, \dots, P_n\seq p_n)\seq u)\in \baseB$ and $T,P_i\deriveBaseIML{\baseB}p_i$ for each $i$. By the inductive hypothesis, we therefore have that $P_i\deriveBaseIML{\baseC}p_i$ for each $i$, from which it follows by (App) that $\deriveBaseIML{\baseC}u$.

Going right to left, we consider whether $u\in T$ or not. If it is, then we immediately have by (Ref) that $T\deriveBaseIML{\baseB}u$ and we are done. So suppose not. We begin this case by making the observation that if for any base $\baseX$ we have that $Q\deriveBaseIML{\baseX}q$ holds by (Ref) then it must be the case that $Q\deriveBaseIML{\baseB}q$ and by Lemma~\ref{lem:basicWeakening} we have that $T,Q\deriveBaseIML{\baseB}q$. 

Now consider the base $\base{S} = \baseB \cup \{\seq t \,|\, t \in T\} \baseGeq \baseB$. We deduce that $\deriveBaseIML{\base{S}}u$ holds and that it must hold by (App) or if $u$ is a labelled formula it can hold by a modal case. Starting with (App), since $u \notin T$, there must be a rule concluding $u$ in $\baseB$. Let this rule be $(P_1\seq p_1, \dots, P_n\seq p_n)\seq u$. Thus, we have by (App) that $P_i\deriveBaseIML{\base{S}}p_i$ for all $i$. Thus, by the inductive hypothesis, taking our initial observation as our base case, we have that $T,P_i\deriveBaseIML{\baseB}p_i$ from which, by (App), we conclude $T\deriveBaseIML{\baseB}u$, as required. 

We now show just one of the modal cases --- all the others go similarly. Suppose $u=p^x$ and that we have that $\deriveBaseIML{\base{S}}p^x$ holds by ${\rm (D)}$. Thus, we have that $\gamma_D\in\gamma$ and that there exists a $y$ such that $xRy\deriveBaseIML{\base{S}}p^x$. By the inductive hypothesis, possibly by renaming the variable $y$, we therefore have that there exists a $y$ such that $T,xRy\deriveBaseIML{\baseB}p^x$, and by ${\rm (D)}$ we obtain $T\deriveBaseIML{\baseB}p^x$, as required. \fillBox
\end{proof}

\begin{lemma}[Monotonicity]
    If $S\deriveBaseIML{\baseB}p$, then, for all $\baseC\baseGeq\baseB$, 
    $S\deriveBaseIML{\baseC}p$. \fillBox
\end{lemma}
To prove this, it suffices to consider how $S\deriveBaseIML{\baseB}p$ holds. Intuitively, we understand that by extending our base, the original atomic derivation still remains a valid argument for $p$ from $S$.

With these preliminary notions established, we can extend the notion of basic derivability to give a support relation, in the sense of Sandqvist \cite{Tor2015}. First, must extend the notion of sequent to allow for relational assumptions or, in the presence of no assumptions, concluding a relation.

\begin{definition}[Extended sequent]
    An extended sequent is a pair $(\Gamma,\varphi)$ that we write $(\Gamma:\varphi)$ where $\Gamma$ is a set of intuitionistic modal formulae and relational assumptions and $\varphi$ is an intuitionistic modal formula or, if $\Gamma$ is empty, also possibly a relational assumption. \fillBox
\end{definition}

The support relation, Definition~\ref{def:support} is core to the base-extension semantics. 

\begin{definition}[Support] \label{def:support}
    Given a set of frame conditions $\gamma$ and a base $\baseB$, we can inductively define a relation of support on extended sequents as follows:
\[
\begin{array}{lrcl}
\mbox{\rm (At)} & \mbox{$\suppIML{\baseB}p^x$} & \;\mbox{iff}\; & 
\mbox{$\deriveBaseIML{\baseB}p^x$} \\ 
\mbox{\rm (Rel)} & \mbox{$\suppIML{\baseB}xRy$} & \;\mbox{iff}\; & \mbox{$\deriveBaseIML{\baseB}xRy$} \\ 
\mbox{$(\land)$} & \mbox{$\suppIML{\baseB}(\varphi\land\psi)^x$} & \;\mbox{iff}\; & \mbox{$\suppIML{\baseB}\varphi^x$ and $\suppIML{\baseB}\psi^x$} \\ 
\mbox{$(\lor)$} & \mbox{$\suppIML{\baseB}(\varphi\lor\psi)^x$} & \;\mbox{iff}\; & \mbox{for all $\baseC\baseGeq\baseB$ and $p^z$,} \\ 
& & & \mbox{if $\varphi^x\suppIML{\baseC}p^z$ and $\psi^x\suppIML{\baseC}p^z$, then $\suppIML{\baseC}p^z$} 
\end{array}
\]
\[
\begin{array}{lrcl}
\mbox{$(\to)$} & \mbox{$\suppIML{\baseB}(\varphi\to\psi)^x$} & \;\mbox{iff}\; & \mbox{$\varphi^x\suppIML{\baseB}\psi^x$} \\ 
\mbox{$(\Box)$} & \mbox{$\suppIML{\baseB}(\Box\varphi)^x$} & \;\mbox{iff}\; & \mbox{$xRy\suppIML{\baseB}\varphi^y$, for all labels $y$} \\ 
\mbox{$(\Diamond)$} & \mbox{$\suppIML{\baseB}(\Diamond\varphi)^x$} & \;\mbox{iff}\; & \mbox{for all $\baseC\baseGeq\baseB$ and $p^z$,} \\ 
& & & \mbox{if $xRy,\varphi^y\suppIML{\baseC}p^z$ for all labels $y$, then $\suppIML{\baseC}p^z$} 
\end{array}
\]
\[
\begin{array}{lrcl} 
\mbox{$(\bot)$} & \mbox{$\suppIML{\baseB}\bot^x$} & \;\mbox{iff}\; & \mbox{$\suppIML{\baseB}p^z$, for all $p^z$}\\
\mbox{$(\top)$} & \mbox{$\suppIML{\baseB}\top^x$} &  & \mbox{always}\\
\mbox{\rm (Inf)} & \mbox{$\Gamma\suppIML{\baseB}\varphi$} & \;\mbox{iff}\; & \mbox{for all $\baseC\baseGeq\baseB$ such that $\suppIML{\baseC}\psi$, for each $\psi\in\Gamma$,} \\  & & & \mbox{implies $\suppIML{\baseC}\varphi$}
\end{array} \vspace{-5mm}
\]
\fillBox
\end{definition}

\begin{definition}[Validity]
    The sequent $(\Gamma:\varphi^x)$ is said to be valid if and only if for all bases $\baseB$, it is the case that $\Gamma\suppIML{\baseB}\varphi^x$ holds. \fillBox
\end{definition}
We start with an ostensibly useful observation.
\begin{lemma}
    The sequent $(\Gamma:\varphi^x)$ is valid if and only if $\Gamma\suppIML{\emptyset}\varphi^x$
\end{lemma}
\begin{proof}
    Going left to right, we have that for all bases $\baseB$, it is the case that $\Gamma\suppIML{\baseB}\varphi^x$. Thus, we can pick $\baseB=\emptyset$. Going right to left, we obtain the result by monotonicity as the empty base is the smallest subset of every base. \fillBox
\end{proof}

So, we are justified in writing a valid sequent as $\Gamma\suppM{}{\,\gamma}\varphi^x$. Before moving on, a brief discussion of the clause for $\Diamond$ is in order. One would perhaps wish to see a clause for $\Diamond$ that says something along the lines of $\suppIML{\baseB}(\Diamond\varphi)^x$ iff $\exists y (\suppIML{\baseB}xRy$ and $\suppIML{\baseB}\varphi^y)$. Doing so, however, quickly presents problems in proving both soundness and completeness. If we understand the presence of the existential as an infinite, meta-level disjunction, then we can understand this problem as being of the same kind faced by Sandqvist \cite{Tor2015} in trying to define disjunction. In fact, the clause for $\Diamond$ amounts to an infinitary version of Sandqvist's clause for disjunction. 

More specifically, suppose one gives a na\"{\i}ve Kripke-like definition for intuitionistic disjunction; that is,  $\Vdash_{\baseB}\varphi\lor\psi\text{ iff }\Vdash_{\baseB}\varphi\text{ or }\Vdash_{\baseB}\psi$. One could instead write this definition as $\exists\chi\in\Delta$ s.t. $\Vdash_{\baseB}\chi$, where $\Delta=\{\varphi,\psi\}$. This definition, as shown by de Campos Sanz and Piecha in~\cite{CamposSanz2014-CAMACR-5}, doesn't work. 

Instead, the clause $\forall \baseC\baseGeq\baseB,\forall p,\forall\chi\in\Delta$, $\chi\Vdash_{\baseC}p$ implies $ \Vdash_{\baseC}p$ --- corresponding to the second-order definition of $\vee$ or to the elimination rule in NJ --- does (again, see \cite{GP2025} for more discussion). Something similar follows for $\Diamond$. If one attempts to give $\Diamond$ a Kripke-like definition, that is, $\suppIML{\baseB}(\Diamond\varphi)^x$ iff $\exists y$ s.t.  $(\suppIML{\baseB}xRy$ and $\suppIML{\baseB}\varphi^y)$, then, similarly to before, this definition fails to be sound and complete. Instead, as demonstrated below, the clause $\forall \baseC\baseGeq\baseB,\forall p^z,\forall y\in\Worlds(\varphi^y,xRy\Vdash_{\baseC}p^z)$ implies $ \Vdash_{\baseC}p^z$, again based on the elimination rule, is sound and complete. We note that, in both cases, we have taken definitions that are meta-level disjunctive in the Kripke semantics, and turned them around into meta-level conjunctive in their proof-theoretic semantics.

\section{Soundness} \label{sec:soundness}

The soundness of the base-extension semantics we have given with respect to natural deduction system $\rm N_{\Box\Diamond}(\gamma)$ is stated
as follows:
\[
 \mbox{if $\Gamma\vdash^{\,\gamma}\varphi^x$, then $\Gamma\suppM{}{\,\gamma}\varphi^x$.}
\]
By the inductive definition of derivability in $\rm N_{\Box\Diamond}(\gamma)$, it suffices to show the following:
    \begin{itemize}
        \item[$(\text{R})$] $\Gamma,\varphi^x\suppM{}{\gamma}\varphi^x$
        \item[$(\land\text{I})$] If $\Gamma\suppM{}{\gamma}\varphi^x$ and $\Gamma\suppM{}{\gamma}\psi^x$, then $\Gamma\suppM{}{\gamma}(\varphi\land\psi)^x$
        \item[$(\land\text{E})$] If $\Gamma\suppM{}{\gamma}(\varphi\land\psi)^x$, then $\Gamma\suppM{}{\gamma}\varphi^x$ and $\Gamma\suppM{}{\gamma}\psi^x$
        \item[$(\lor\text{I})$] If $\Gamma\suppM{}{\gamma}\varphi^x$ or $\Gamma\suppM{}{\gamma}\psi^x$, then $\Gamma\suppM{}{\gamma}(\varphi\lor\psi)^x$
        \item[$(\lor\text{E})$] If $\Gamma\suppM{}{\gamma}(\varphi\lor\psi)^x$ and $\Gamma,\varphi^x\suppM{}{\gamma}\chi^y$ and $\Gamma,\psi^x\suppM{}{\gamma}\chi^y$, then $\Gamma\suppM{}{\gamma}\chi^y$
        \item[$(\to\text{I})$] If $\Gamma,\varphi^x\suppM{}{\gamma}\psi^x$, then $\Gamma\suppM{}{\gamma}(\varphi\to\psi)^x$
        \item[$(\to\text{E})$] If $\Gamma\suppM{}{\gamma}(\varphi\to\psi)^x$ and $\Gamma\suppM{}{\gamma}\varphi^x$, then $\Gamma\suppM{}{\gamma}\psi^x$
        \item[$(\bot\text{E})$] If $\Gamma\suppM{}{\gamma}\bot^x$, then $\Gamma\suppM{}{\gamma}\chi^y$
        \item[$(\Box\,\text{I})$] If $\Gamma,xRy\suppM{}{\gamma}\varphi^y$ for $y$ different to $x$ and not a label of any element of $\Gamma$, then $\Gamma\suppM{}{\gamma}(\Box\varphi)^x$.
        \item[$(\Box\,\text{E})$] If $\Gamma\suppM{}{\gamma}(\Box\varphi)^x$ and $\Gamma\suppM{}{\gamma}xRy$, then $\Gamma\suppM{}{\gamma}\varphi^y$ 
        \item[$(\Diamond\text{I})$] If $\Gamma\suppM{}{\gamma}\varphi^v$ and $\Gamma\suppM{}{\gamma}xRv$, then $\Gamma\suppM{}{\gamma}(\Diamond \varphi)^x$
        \item[$(\Diamond\text{E})$] If $\Gamma\suppM{}{\gamma}(\Diamond\varphi)^x$ and $\Gamma,\varphi^a,xRa\suppM{}{\gamma}\psi^z$ for $a$ different to $x$ and $z$ and not a label of any element of $\Gamma$, then $\Gamma\suppM{}{\gamma}\psi^z$.
    \end{itemize}
    Additionally, if $\gamma\neq\emptyset$ we must show the following:
    \begin{itemize}
        \item[$(R_D)$] If $\gamma_D\in\gamma$ and $\Gamma,xRa\suppM{}{\gamma}\varphi^z$, for $a$ different to $x$ and $z$ and not a label of any element of $\Gamma$, then $\Gamma\suppM{}{\gamma}\varphi^z$
        \item[$(R_T)$] If $\gamma_T\in\gamma$ and $\Gamma,xRx\suppM{}{\gamma}\varphi^y$ then $\Gamma\suppM{}{\gamma}\varphi^y$
        \item[$(R_B)$] If $\gamma_B\in\gamma$ and $\Gamma\suppM{}{\gamma}xRy$ and $\Gamma,yRx\suppM{}{\gamma}\varphi^z$ then $\Gamma\suppM{}{\gamma}\varphi^z$
        \item[$(R_4)$] If $\gamma_4\in\gamma$ and $\Gamma \suppM{}{\gamma} xRy$ and $\Gamma\suppM{}{\gamma}yRz$ and $\Gamma,xRz\suppM{}{\gamma}\varphi^w$ then $\Gamma\suppM{}{\gamma}\varphi^w$
        \item[$(R_5)$] If $\gamma_5\in\gamma$ and $\Gamma \suppM{}{\gamma} xRy$ and $\Gamma\suppM{}{\gamma}xRz$ and $\Gamma,yRz\suppM{}{\gamma}\varphi^w$ then $\Gamma\suppM{}{\gamma}\varphi^w$
        \item[$(R_2)$] If $\gamma_2\in\gamma$ and $\Gamma \suppM{}{\gamma} xRy$ and $\Gamma\suppM{}{\gamma}xRz$ and $\Gamma,yRw,zRw\suppM{}{\gamma}\varphi^v$, for $w$ different to $x,y,z,v$ and not a label of any element of $\Gamma$, then $\Gamma\suppM{}{\gamma}\varphi^v$.
    \end{itemize}

\begin{theorem}[Soundness]
    If $\Gamma\vdash^{\gamma}\varphi^x$, then $\Gamma\suppM{}{\,\gamma}\varphi^x$.
\end{theorem}

\begin{proof}
   
    We proceed by proving each of the cases listed above, under the additional hypothesis that we are in an arbitrary base $\baseB$ such that for all $\theta\in\Gamma$ we have that $\suppIML{\baseB}\theta$.
    \begin{itemize}
        \item[$(\text{R})$] $\varphi^x\suppIML{\baseB}\varphi^x$. This case is immediate by (Inf).
        \item[$(\land\text{I})$] If $\suppIML{\baseB}\varphi^x$ and $\suppIML{\baseB}\psi^x$, then $\suppIML{\baseB}(\varphi\land\psi)^x$. This case follows immediately by the definition of $(\land)$.
        \item[$(\land\text{E})$] If $\suppIML{\baseB}(\varphi\land\psi)^x$, then $\suppIML{\baseB}\varphi^x$ and $\suppIML{\baseB}\psi^x$. This case also follows immediately by the definition of $(\land)$.
        \item[$(\lor\text{I})$] If $\suppIML{\baseB}\varphi^x$ or $\suppIML{\baseB}\psi^x$, then $\suppIML{\baseB}(\varphi\lor\psi)^x$. To show this case we suppose we have, for all $\baseC\baseGeq\baseB$ and basic sentences $p^y$, that $\varphi^x\suppIML{\baseC}p^y$ and $\psi^x\suppIML{\baseC}p^y$. We must show that $\suppIML{\baseC}p$. From the hypothesis, we have that $\suppIML{\baseB}\varphi^x$ or $\suppIML{\baseB}\psi^x$, which, by monotonicity, gives that $\suppIML{\baseC}\varphi^x$ or $\suppIML{\baseC}\psi^x$. In either case, by (Inf), we therefore obtain that $\suppIML{\baseC}p^y$, as required.
        \item[$(\lor\text{E})$] If $\suppIML{\baseB}(\varphi\lor\psi)^x$ and $\varphi^x\suppIML{\baseB}\chi^y$ and $\psi^x\suppIML{\baseB}\chi^y$, then $\suppIML{\baseB}\chi^y$. We proceed by an inductive argument over the structure of the labelled formula $\chi^y$.
            \begin{itemize}[leftmargin=*]
                \item[-]$\chi = p$. In this case, we have sufficient hypotheses by the definition of $(\lor)$ to conclude $\suppIML{\baseB}p^y$, as required.
                \item[-]$\chi = \alpha\land\beta$. In this case, by the definition of $(\land)$, we have that $\varphi^x\suppIML{\baseB}\alpha^y$ and $\varphi^x\suppIML{\baseB}\beta^y$, and similarly for  $\psi^x$. Thus, by the inductive hypothesis, we have that $\suppIML{\baseB}\alpha^y$ and $\suppIML{\baseB}\beta^y$, and so we are done.
                \item[-]$\chi = \alpha\lor\beta$. In this case, we start by further supposing that we are in a base $\baseC\baseGeq\baseB$ and that we have, for all atoms $p^z$, that $\alpha^y\suppIML{\baseC}p^z$ and $\beta^y\suppIML{\baseC}p^z$. Our goal will be to show that $\suppIML{\baseC}p^z$. We want to show that $\varphi^x\suppIML{\baseC}p$ and $\psi^x\suppIML{\baseC}p^z$, which will give us sufficient grounds to then use our first hypothesis to give us our conclusion. We will only show $\varphi^x\suppIML{\baseC}p^z$ as the argument is similar for the other case. We start by observing that we have, by hypothesis, $\varphi^x\suppIML{\baseB}(\alpha\lor\beta)^y$. By monotonicity, this is $\varphi^x\suppIML{\baseC}(\alpha\lor\beta)^y$. By (Inf), this gives that, for all $\baseD\baseGeq\baseC$, $\suppIML{\baseD}\varphi^x$ implies $\suppIML{\baseD}(\alpha\lor\beta)^y$. Since we have that $\alpha^y\suppIML{\baseC}p^z$ and $\beta^y\suppIML{\baseC}p^z$, by hypothesis, then we  conclude that we have, for all $\baseD\baseGeq\baseC$, if $\suppIML{\baseD}\varphi^x$, then $\suppIML{\baseD}p^z$. By (Inf), this gives $\varphi^x\suppIML{\baseC}p^z$, as required. 
                \item[-]$\chi = \alpha\to\beta$. In this case, we further suppose we are in a base $\baseC\baseGeq\baseB$ such that $\suppIML{\baseC}\alpha^y$. Our goal will be to show that $\suppIML{\baseC}\beta^y$. We start by considering the hypothesis $\varphi^x\suppIML{\baseB}(\alpha\to\beta)^y$. By monotonicity, we have that $\varphi^x\suppIML{\baseB}(\alpha\to\beta)^y$ and that, by (Inf), for all $\baseD\baseGeq\baseC$, if $\suppIML{\baseD}\varphi^x$, then $\suppIML{\baseD}(\alpha\to\beta)^y$. Since we have that $\suppIML{\baseC}\alpha^y$ by hypothesis, then, by monotonicity and (Inf), we conclude that we have for all $\baseD\baseGeq\baseC$, if $\suppIML{\baseD}\varphi^x$, then $\suppIML{\baseD}\beta^y$, which gives us $\varphi^x\suppIML{\baseC}\beta^y$. Similarly for $\psi^x$, at which point, by the inductive hypothesis, we can conclude that $\suppIML{\baseC}\beta^y$, as required.
                \item[-]$\chi = \bot$. This case is immediate by the definition of $\bot^x$.
                \item[-]$\chi=\Box \alpha$. Our goal will be to show that, $yRa\suppIML{\baseB}\alpha^a$ holds, for all labels $a$. With this in mind, we fix an arbitrary label $a$ and consider all bases $\baseC\baseGeq\baseB$, where $\suppIML{\baseC}yRa$ holds. We are left with showing $\suppIML{\baseC}\alpha^a$.
                To this end, consider the hypothesis $\varphi^x\suppIML{\baseB}(\Box\alpha)^y$. We immediately have that $\varphi^x\suppIML{\baseC}(\Box\alpha)^y$ by monotonicity. This is equivalent to considering all bases $\baseD\baseGeq\baseC$ for which $\suppIML{\baseD}\varphi^x$ implies $\suppIML{\baseD}(\Box\alpha)^y$. 
                The consequent of the previous implication is equivalent to considering all labels $b$ such that $yRb\suppIML{\baseD}\alpha^b$ which implies that in particular $yRa\suppIML{\baseD}\alpha^a$. Since, by hypothesis, we have that $\suppIML{\baseC}yRa$ and $\baseD\baseGeq\baseC$, then it is the case that $\suppIML{\baseD}\alpha^a$. It thus follows that $\varphi^x\suppIML{\baseC}\alpha^a$. The same argument gives that $\psi^x\suppIML{\baseC}\alpha^a$, so, by the inductive hypothesis, we can therefore obtain $\suppIML{\baseC}\alpha^a$, as required.
                \item[-]$\chi=\Diamond \alpha$. We start by fixing a base $\baseC\baseGeq\baseB$ and an atom $p^w$, where $yRz,\alpha^z\suppIML{\baseC}p^w$, for all $z$. We must show that $\suppIML{\baseC}p^w$. The second hypothesis, by monotonicity, gives $\varphi^x\suppIML{\baseC}(\Diamond\alpha)^y$. 
                This equivalently gives us that, for all bases $\baseD\baseGeq\baseC$, $\suppIML{\baseD}\varphi^x$ implies that, for all bases $\baseE\baseGeq\baseD$ and atoms $p^a$, if $yRz,\alpha^z\suppIML{\baseE}p^a$, for all $z$, then $\suppIML{\baseE}p^a$. Thus, we can conclude that $\suppIML{\baseD}p^w$ and therefore $\varphi^x\suppIML{\baseC}p^w$. Similarly, we obtain $\psi^x\suppIML{\baseC}p^w$, and so, by the first hypothesis, we conclude $\suppIML{\baseC}p^w$, as required.
            \end{itemize}
        \item[$(\to\text{I})$] If $\varphi^x\suppIML{\baseB}\psi^x$, then $\suppIML{\baseB}(\varphi\to\psi)^x$. This case follows immediately by the definition of $(\to)$.
        \item[$(\to\text{E})$] If $\suppIML{\baseB}(\varphi\to\psi)^x$ and $\suppIML{\baseB}\varphi^x$ then $\suppIML{\baseB}\psi^x$. In this case, the first hypothesis, by (Inf), is equivalent to saying that for all bases $\baseC\baseGeq\baseB$, if $\suppIML{\baseC}\varphi^x$ then $\suppIML{\baseC}\psi^x$. By hypothesis, we have that $\suppIML{\baseB}\varphi^x$ and thus it follows that $\suppIML{\baseB}\psi^x$, as required.
        \item[$(\bot\text{E})$] If $\suppIML{\baseB}\bot^x$, then $\suppIML{\baseB}\chi^y$. Since we have $\suppIML{\baseB}p^z$, for all atoms $p^z$, then by a simple inductive argument over the structure of $\chi^y$, we obtain $\suppIML{\baseB}\chi^y$, as required.
        \item[$(\Box\,\text{I})$] If $xRy\suppIML{\baseB}\varphi^y$ for $y$ different to $x$, then $\suppIML{\baseB}(\Box\varphi)^x$. To show the conclusion, it suffices, by the definition of $(\Box)$, to show that, for all $z$, we have $xRz\suppIML{\baseB}\varphi^z$. Thus, it is required to show that $xRv\suppIML{\baseB}\varphi^v$, for some $v$ not equal to $x$, which we have, by hypothesis, by setting $v=y$. 
        \item[$(\Box\,\text{E})$] If $\suppIML{\baseB}(\Box\varphi)^x$ and $\suppIML{\baseB}xRy$, then $\suppIML{\baseB}\varphi^y$. This case is immediate, since we consider the definition of $\suppIML{\baseB}(\Box\varphi)^x$ at the label $y$ to obtain the result.
        \item[$(\Diamond\text{I})$] If $\suppIML{\baseB}\varphi^v$ and $\suppIML{\baseB}xRv$, then $\suppIML{\baseB}(\Diamond \varphi)^x$. To show $\suppIML{\baseB}(\Diamond \varphi)^x$, we start by additionally fixing a $\baseC\baseGeq\baseB$ and a $p^w$ such that $xRy,\varphi^y\suppIML{\baseC}p^w$ holds for all $y$ and show that $\suppIML{\baseC}p^w$. We have, by monotonicity, that both $\suppIML{\baseC}\varphi^v$ and $\suppIML{\baseC}xRv$ hold. Furthermore, since $xRy,\varphi^y\suppIML{\baseC}p^w$ holds for all $y$, then in particular, we have that $xRv,\varphi^v\suppIML{\baseC}p^w$. Therefore, we conclude $\suppIML{\baseC}p^w$, as required.
        \item[$(\Diamond\text{E})$] If $\suppIML{\baseB}(\Diamond\varphi)^x$ and $\varphi^a,xRa\suppIML{\baseB}\psi^z$ for $a$ different to $x$ and $z$, then $\suppIML{\baseB}\psi^z$. We break the proof into a case analysis over the structure of $\psi^z$.
            \begin{itemize}[leftmargin=*]
                \item[-]$\psi = p$. We start by considering the second hypothesis. Since $a$ is a label different to $x$ and $z$ and didn't appear in any open assumptions, we conclude that $xRy,\varphi^y\suppIML{\baseB}p^z$ for all labels $y$. We thus have sufficient grounds to conclude, by the first hypothesis, that $\suppIML{\baseB}p^z$, as required.
                \item[-]$\psi = \alpha\to\beta$. In this case, we start by fixing a $\baseC\baseGeq\baseB$ such that $\suppIML{\baseC}\alpha^z$. We want to show that $\suppIML{\baseC}\beta^z$. By monotonicity, the second hypothesis therefore gives that $\varphi^a,xRa\suppIML{\baseC}\beta^z$ from which it follows that $\varphi^y,xRy\suppIML{\baseC}\beta^z$ for all labels $y$. Since, by monotonicity, we have that $\suppIML{\baseC}(\Diamond\varphi)^x$, by applying the inductive hypothesis, we therefore obtain that $\suppIML{\baseC}\beta^z$, as required.
                \item[-]$\psi = \alpha\land\beta$. In this case, the second hypothesis gives that $\varphi^a,xRa\suppIML{\baseB}(\alpha\land\beta)^z$, which equivalently gives $\varphi^a,xRa\suppIML{\baseB}\alpha^z$ and $\varphi^a,xRa\suppIML{\baseB}\beta^z$, and therefore, by the arbitrariness of $a$, that $\varphi^y,xRy\suppIML{\baseB}\alpha^z$ and $\varphi^y,xRy\suppIML{\baseB}\beta^z$, for all $y$. Thus, by the inductive hypothesis, we obtain that $\suppIML{\baseB}\alpha^z$ and $\suppIML{\baseB}\beta^z$, which gives $\suppIML{\baseB}(\alpha\land\beta)^z$, as required.
                \item[-]$\psi = \alpha\lor\beta$. In this case, further fix a base $\baseC\baseGeq\baseB$ and an atom $p^w$ such that $\alpha^z\suppIML{\baseC}p^w$ and $\beta^z\suppIML{\baseC}p^w$. We want to show that $\suppIML{\baseC}p^w$. By monotonicity, taking the second hypothesis is equivalent to considering all $\baseD\baseGeq\baseC$ such that $\suppIML{\baseD}\varphi^a$ and $\suppIML{\baseD}xRa$ implies $\suppIML{\baseD}(\alpha\lor\beta)^z$. By definition, this is equivalent to considering all bases $\baseE\baseGeq\baseD$ and atoms $p^w$ such that $\alpha^z\suppIML{\baseE}p^w$ and $\beta^z\suppIML{\baseE}p^w$ implies $\suppIML{\baseE}p^w$. Since we have that $\alpha^z\suppIML{\baseC}p^w$ and $\beta^z\suppIML{\baseC}p^w$, then we can conclude that we have $\suppIML{\baseD}p^w$. Thus we have that $xRa,\varphi^a\suppIML{\baseC}p^w$, which implies that $xRy,\varphi^y\suppIML{\baseC}p^w$, for all $y$. Hence, by the first hypothesis, we get $\suppIML{\baseC}p^w$, as required.
                \item[-]$\psi = \bot$. In this case, it suffices to show that $\suppIML{\baseB}p^b$, for an arbitrary atoms $p^b$. We note that the second hypothesis implies that $xRa,\varphi^a\suppIML{\baseC}p^b$ from which we conclude, by the arbitrariness of $a$, that $xRy,\varphi^y\suppIML{\baseB}p^w$, for all $y$. Hence, by the first hypothesis, we conclude $\suppIML{\baseB}p^b$, as required.
                \item[-]$\psi = \Box\alpha$. In this case, further fix a label $v$ and a base $\baseC\baseGeq\baseB$, such that $\suppIML{\baseC}zRv$. We are left to show that $\suppIML{\baseC}\alpha^v$. By monotonicity, the second hypothesis gives, for all $\baseD\baseGeq\baseC$, that $\suppIML{\baseD}xRa$ and $\suppIML{\baseD}\varphi^a$ implies that, for all $y$, we have $zRy\suppIML{\baseD}\alpha^y$. Therefore, in particular, we have that $zRv\suppIML{\baseD}\alpha^v$. Since $\suppIML{\baseC}zRv$, we therefore conclude that $xRa,\varphi^a\suppIML{\baseC}\alpha^v$. By the arbitrariness of $a$, we therefore have that $xRa,\varphi^a\suppIML{\baseC}\alpha^v$, for all $y$. Since, by monotonicity, we have that $\suppIML{\baseC}(\Diamond\varphi)^x$, by applying the inductive hypothesis, we therefore obtain $\suppIML{\baseC}\alpha^v$, as required.
                \item[-]$\psi = \Diamond\alpha$. In this case, further fix a base $\baseC\baseGeq\baseB$ and an atom $p^w$ for which $zRv, \varphi^v\suppIML{\baseC}p^w$ holds for all $v$. We want to show that $\suppIML{\baseC}p^w$. By monotonicity, the second hypothesis gives that, for all $\baseD\baseGeq\baseC$, $\suppIML{\baseD}\varphi^a$ and $\suppIML{\baseD}xRa$ implies $\suppIML{\baseD}(\Diamond\alpha)^z$. The conclusion of the previous implication is equivalent to considering all bases $\baseE\baseGeq\baseD$ and atoms $p^b$ such that if $zRv, \varphi^v\suppIML{\baseE}p^b$ for all $v$, then $\suppIML{\baseE}p^b$. Thus we conclude that $\suppIML{\baseD}p^w$ and so that $xRa,\varphi^a\suppIML{\baseC}p^w$. Therefore, we have that $xRy,\varphi^y\suppIML{\baseC}p^w$ for all $y$ from which, by the first hypothesis, we conclude that $\suppIML{\baseC}p^w$, as required.
            \end{itemize}
    \end{itemize}
    Now we consider the cases where $\gamma\neq\emptyset$. In each case, we assume that $\gamma$ contains that particular condition. We leave the proofs of all cases except for one as an exercise as they all follow exactly the same pattern. 
    \begin{itemize}[leftmargin=*]
        \item[$(R_D)$] If $xRa\suppIML{\baseB}\varphi^z$, for $a$ different to $x$, then $\suppIML{\baseB}\varphi^z$. We proceed by induction on the structure of $\varphi^z$.
        \begin{itemize}[leftmargin=*]
            \item[-] If $\varphi=p$, then we start by observing that there exists a $y$ such that $xRy\suppIML{\baseB}\varphi^z$. Thus, we are immediately done by the definition of $\deriveBaseIML{\baseB}$.
            \item[-] If $\varphi=\alpha\land\beta$, then we have that $xRa\suppIML{\baseB}\alpha^z$ and $xRa\suppIML{\baseB}\beta^z$. By the inductive hypothesis, it follows that $\suppIML{\baseB}\alpha^z$ and $\suppIML{\baseB}\beta^z$, and so we can conclude $\suppIML{\baseB}(\alpha\land\beta)^z$, as required.
            \item[-] If $\varphi= \alpha\to\beta$, then we have that $xRa\suppIML{\baseB}(\alpha\to\beta)^z$. We must show that  
            $\suppIML{\baseB}(\alpha\to\beta)^z$. With this in mind, we fix a base $\baseC\baseGeq\baseB$, where $\suppIML{\baseC}\alpha^z$, and show that $\suppIML{\baseC}\beta^z$ holds. 
            By monotonicity, we have that $xRa\suppIML{\baseC}(\alpha\to\beta)^z$, which is equivalent to having, for all $\baseD\baseGeq\baseC$ with $\suppIML{\baseD}xRa$, that $\alpha^z\suppIML{\baseD}\beta^z$. But since $\baseD\baseGeq\baseC$, where $\suppIML{\baseC}\alpha^z$, we therefore have that $\suppIML{\baseD}\beta^z$, and thus $xRa\suppIML{\baseC}\beta^z$, from which we conclude that there exists a $y$ such that $xRy\suppIML{\baseC}\beta^z$. By the inductive hypothesis, we therefore have that $\suppIML{\baseC}\beta^z$, as required. 
            \item[-] If $\varphi = \alpha\lor\beta$, then we have that $xRa\suppIML{\baseB}(\alpha\lor\beta)^z$. Given a base $\baseC\baseGeq\baseB$ and an atom $p^w$ such that $\varphi^z\suppIML{\baseC}p^w$ and $\psi^z\suppIML{\baseC}p^w$, we want to show that $\suppIML{\baseC}p^w$. 
            We start by considering that, by monotonicity, we have $xRa\suppIML{\baseC}(\alpha\lor\beta)^z$, which is equivalent to considering all bases $\baseD\baseGeq\baseC$ for which $\suppIML{\baseD}xRa$ implies $\suppIML{\baseD}(\alpha\lor\beta)^z$. The conclusion of the previous implication is equivalent to considering all bases $\baseE\baseGeq\baseD$ and atoms $q^v$, where $\varphi^z\suppIML{\baseE}q^v$ and $\psi^z\suppIML{\baseE}q^v$ implies $\suppIML{\baseE}q^v$. 
            Since we have by hypothesis that $\varphi^z\suppIML{\baseC}p^w$,  $\psi^z\suppIML{\baseC}p^w$, and $\baseE\baseGeq\baseC$, we therefore conclude that $xRa\suppIML{\baseC}p^v$, and thus that there exists a $y$ such that $xRa\suppIML{\baseC}p^v$. Then, by the definition of $\deriveBaseIML{\baseC}$, we obtain $\suppIML{\baseC}p^v$, as desired.
            \item[-] If $\varphi = \bot$, then we have that $xRa\suppIML{\baseB}\bot^z$ and we must show that $\suppIML{\baseB}\bot^z$. By definition, for each $p^w$, there exists a $y$ such that $xRy\suppIML{\baseB}p^w$ for any $p^w$, and so, by the definition of $\deriveBaseIML{\baseB}$, we obtain $\suppIML{\baseB}p^w$, for all $p^w$. Therefore $\suppIML{\baseB}\bot^z$.
            \item[-] If $\varphi = \Box\alpha$, then we have that $xRa\suppIML{\baseB}(\Box\alpha)^z$ and we must show, given a base $\baseC\baseGeq\baseB$ such that $\suppIML{\baseC}zRw$ for an arbitrary $w$, that $\suppIML{\baseC}\alpha^w$. By monotonicity, we have $xRy\suppIML{\baseC}(\Box\alpha)^z$, which is equivalent to considering all bases $\baseD\baseGeq\baseC$ for which $\suppIML{\baseD}xRy$ implies $\suppIML{\baseD}(\Box\alpha)^z$. The conclusion of the previous implication is equivalent to considering all labels $v$ such that $zRv\suppIML{\baseD}\alpha^v$. Since we have $\suppIML{\baseC}zRw$, then we have, in particular, that $\suppIML{\baseD}\alpha^v$ and thus $xRa\suppIML{\baseC}\alpha^w$. We therefore have that there exists a $y$ such that $xRa\suppIML{\baseC}\alpha^w$ from which, by the inductive hypothesis, we obtain $\suppIML{\baseC}\alpha^w$, as required.
            \item[-] If $\varphi = \Diamond\alpha$, then we have that $xRa\suppIML{\baseB}(\Diamond\alpha)^z$. To show this, we begin by fixing a base $\baseC\baseGeq\baseB$ and an atom $p^w$ such that $zRa,\alpha^a\suppIML{\baseC}p^w$ for all a, with the goal of showing that $\suppIML{\baseC}p^w$. By monotonicity, we have that $xRa\suppIML{\baseC}(\Diamond\alpha)^z$, which gives that, for all $\baseD\baseGeq\baseC$, we have that $\suppIML{\baseD}xRy$ implies, for all $\baseE\baseGeq\baseD$ and all atoms $q^v$, if $zRb,\alpha^b\suppIML{\baseE}q^v$ for all $b$, then $\suppIML{\baseE}q^v$. Since we have $zRa,\alpha^a\suppIML{\baseC}p^w$, we therefore can obtain by monotonicity that $\suppIML{\baseD}p^w$ and therefore $xRa\suppIML{\baseC}p^w$, and thus that there exists a $y$ such that $xRy\suppIML{\baseC}p^w$. Thus, by the definition of $\deriveBaseIML{\baseC}$, we conclude $\suppIML{\baseC}p^w$, as desired.
        \end{itemize}
        \item[$(R_T)$] If $xRx\suppIML{\baseB}\varphi^y$, then $\suppIML{\baseB}\varphi^y$. 
        \item[$(R_B)$] If $\suppIML{\baseB}xRy$ and $yRx\suppIML{\baseB}\varphi^z$, then $\suppIML{\baseB}\varphi^z$. 
        \item[$(R_4)$] If $\suppIML{\baseB}xRy$, $\suppIML{\baseB}yRz$, and $xRz\suppIML{\baseB}\varphi^z$, then $\suppIML{\baseB}\varphi^v$.
        \item[$(R_5)$] If $\suppIML{\baseB}xRy$, $\suppIML{\baseB}xRz$, and $yRz\suppIML{\baseB}\varphi^v$, then $\suppIML{\baseB}\varphi^v$.
        \item[$(R_2)$] If $\suppIML{\baseB}xRy$, $\suppIML{\baseB}yRz$, and $yRw,zRv\suppIML{\baseB}\varphi^w$, then $\suppIML{\baseB}\varphi^w$, for $w$ different to $x,y,z,v$.
    \end{itemize}
    This concludes the proof of soundness. \fillBox
\end{proof}

\section{Completeness} \label{sec:completeness}

We now show that, given an arbitrary valid sequent $(\Gamma:\varphi^x)$, there exists an $\rm N_{\Box\Diamond}(\gamma)$ proof of it. To show this, we will construct a special base, called $\base{N}$, whose rules will mimic the natural deduction rules of $\rm N_{\Box\Diamond}(\gamma)$, with basic sentences playing the role of (or simulating) the subformulae of the arbitrary sequent. Since the base rules can only be in terms of basic sentences and not formulae, we must be careful, as described below, about how we choose the assignment of basic sentences. 

We then show that derivations in $\base{N}$ directly correspond to natural deduction derivations of the formulae being simulated. Thus, in effect, we show that the sequent $(\Gamma:\varphi^x)$ is provable in $\rm N_{\Box\Diamond}(\gamma)$ by constructing the proof. 

To this end, we fix an arbitrary valid sequent $\mathfrak{S}=(\Gamma:\varphi^x)$ and let $\Xi$ be the set of generalized subformulae of the sequent. We define the generalized subformulae of a formula $\varphi$ as follows:
\begin{itemize}
    \item[-]if $\varphi^x$ is an atom $p^x$, then the only generalized subformula of $p^x$ is $p^x$ itself
    \item[-]if $\varphi^x$ is either $\top^x$ or $\bot^x$, then the only generalized subformula of $\varphi^x$ is $\top^x$ or $\bot^x$, respectively
    \item[-]if $\varphi^x$ is $(\alpha\circ\beta)^x$ for $\circ \in \{\land,\lor,\to\}$, then the generalized subformulae of $\varphi^x$ are $\varphi^x$, $\alpha^x$, and $\beta^x$
    \item[-]if $\varphi^x$ is $(\circ\alpha)^x$ for $\circ \in \{\Box,\Diamond\}$, then the generalized subformulae of $\varphi^x$ are $\varphi^x$, $\alpha^z$, and $xRz$, for all $z\in\Worlds$. 
\end{itemize}
Define the set $\text{Lab}(\Xi) = \{a\,|\,\xi^a\in \Xi\}\cup\{a , b \,|\,aRb\in\Xi \}$ as the set of labels of each element of $\Xi$. We define an injection $\flatIML{\cdot}:\Xi \rightarrow \Worlds\times\At$, called the flattening map, such that:
\begin{itemize}
    \item[-] It is the identity map on atoms and on $\top^x$ and $\bot^x$. 
    \item[-] For non-atomic formulae, $\alpha^x$, and relational assumptions $vRw$, it picks an $p^x$, where $p^x \notin \Xi$, and, for all $\alpha^x,\beta^y\in \Xi$, if $\alpha^x\neq\beta^y$, then $\flatIML{\alpha^x}\neq\flatIML{\beta^y}$. 
\end{itemize}
$\flatIML{\cdot}$ has a left inverse $\deflatILL{(\cdot)}$, defined by:
\begin{itemize}
    \item[-] The identity map on basic sentences and on $\top^x$ and $\bot^x$.
    \item[-] The original labelled formula otherwise.
\end{itemize}
These functions are defined to be distributing over sets; that is, given a set of labelled formulae $\Gamma$, we define $\Gamma^\flat = \{\flatIML{\gamma^x}\,|\, \gamma^x \in \Gamma \}$ and $\Gamma^\natural = \{\deflatIML{\gamma^x}\,|\, \gamma^x \in \Gamma \}$. 

We can now define the simulation base $\base{N}$ relative to $\Xi$ and $\flatIML{\cdot}$ according to the rules of Figure~\ref{fig:FlatNatDedIML}, where
\begin{itemize}
\item[-] $p^z$ ranges over all atoms, 
\item[-] $\varphi^x$, $\varphi^z$, and $\psi^x$ range over elements of $\Xi$,
\item[-] the formula $\varphi^y$ in the rule $(\irn\Box^\flat)$ ranges over all elements of $\Xi$ such that $y\neq x$ for each $\flatIML{(\Box\varphi)^x}$,
\item[-] the formula $\varphi^y$ in the rule $(\ern\Diamond^\flat)$ ranges over all elements of $\Xi$ such that $y\neq x$, and $y\neq z$, for each  $\flatIML{(\Diamond\varphi)^x}$ and $p^z$. 
\end{itemize}

\begin{figure}[t]
    \hrule
    \[
    \begin{array}{c@{\qquad}c} 
        \infer[\irn \top^\flat]{~~\flatIML{\top^x}~~}{} & \infer[\ern \bot^\flat]{p^z}{\flatIML{\bot^x}} \\[3mm]
        \infer[\irn \to^\flat]{\flatIML{(\varphi \to \psi)^x}}{\deduce{\flatIML{\psi^x}}{[\flatIML{\varphi^x}]}} & 
        \infer[\ern \to^\flat]{\flatIML{\psi^x}}{\flatIML{(\varphi \to \psi)^x} & \flatIML{\varphi^x}} \\[3mm]
        \infer[\irn \land^\flat]{\flatIML{(\varphi\land\psi)^x}}{\flatIML{\varphi^x} & \flatIML{\psi^x}} & 
        \infer[\ern{\land_1}^\flat]{\flatIML{\varphi^x}}{\flatIML{(\varphi \land \psi)^x}} \quad \infer[\ern{\land_2}^\flat]{\flatIML{\psi^x}}{\flatIML{(\varphi \land \psi)^x}} \\[3mm] 
        \infer[\irn{\lor_1}^\flat]{\flatIML{(\varphi\lor\psi)^x}}{\flatIML{\varphi^x}}\quad \infer[\irn{\lor_2}^\flat]{\flatIML{(\varphi\lor\psi)^x}}{\flatIML{\psi^x}} & 
        \infer[\ern \lor^\flat]{p^z}{\flatIML{(\varphi\lor\psi)^x} & \deduce{p^z}{[\flatIML{\varphi^x}]} & \deduce{p^z}{[\flatIML{\psi^x}]}}\\[3mm]
        \infer[\irn \Box^\flat]{\flatIML{(\Box\varphi)^x}}{\deduce{\flatIML{\varphi^y}}{[\flatIML{xRy}]}}&
        \infer[\ern \Box^\flat]{\flatIML{\varphi^z}}{\flatIML{(\Box\varphi)^x} & \flatIML{xRz}}\\[3mm]
        \infer[\irn \Diamond^\flat]{\flatIML{(\Diamond\varphi)^x}}{\flatIML{\varphi^z} & \flatIML{xRz}} &
        \infer[\ern \Diamond^\flat]{p^z}{\flatIML{(\Diamond\varphi)^x} & \deduce{p^z}{[\flatIML{\varphi^y}]\,[\flatIML{xRy}]}}\\[3mm]
    \end{array}
    \]
    \caption{Simulation base $\base{N}$} \vspace{2mm}
    \hrule
    \label{fig:FlatNatDedIML}
\end{figure}

\begin{figure}[ht]
    \hrule
    \[
    \begin{array}{c@{\qquad}c}
        \infer[\flatIML{R_D}]{p^a}{\deduce{p^a}{[\flatIML{xRy}]}} & 
        \infer[\flatIML{R_T}]{p^a}{\deduce{p^a}{[\flatIML{xRx}]}} \\[3mm]
        \infer[\flatIML{R_B}]{p^a}{\flatIML{xRz} & \deduce{p^a}{[\flatIML{zRx}]}} & 
        \infer[\flatIML{R_4}]{p^a}{\flatIML{xRw} & \flatIML{wRz} & \deduce{p^a}{[\flatIML{xRz}]}} \\[3mm] 
        \infer[\flatIML{R_5}]{p^a}{\flatIML{xRw} & \flatIML{xRz} & \deduce{p^a}{[\flatIML{wRz}]}} &
        \infer[\flatIML{R_2}]{p^a}{\flatIML{xRw} & \flatIML{xRz} & \deduce{p^a}{[\flatIML{wRy}]\,[\flatIML{zRy}]}} \\[3mm] 
    \end{array}
    \] \vspace{-4mm}
    \caption{Rules which extend $\base{N}$ to express the properties of the relational assumptions} \vspace{2mm}
    \hrule
    \label{fig:FlatNatDedProperties}
\end{figure}

If $\gamma\neq\emptyset$, we must also add rules to $\base{N}$ to express the frame conditions. We do this by adding the rules $\flatIML{R_i}$ from Figure~\ref{fig:FlatNatDedProperties} to $\base{N}$, for each axiom $\gamma_i\in\gamma$ as follows:
\begin{itemize}
    \item If the rule does not contain a relation with the letter $y$, then we add to $\base{N}$ the rule concluding each $p^a$, where all labels range over $\text{Lab}(\Xi)$.
    \item If the rule contains a relation with the letter $y$, then we add to $\base{N}$ the rule concluding each $p^a$, where the label $y$ ranges over all $\Worlds$ and all other labels range over $\text{Lab}(\Xi)$ with the condition that $y$ is never equal to any of the other labels in that rule, that is to say, $y\neq x,a$ in the case of $\flatIML{R_D}$ and $y\neq x,z,v,w,a$ in the case of $\flatIML{R_2}$.
\end{itemize}

We prove completeness, making use of two lemmas that we will prove later in this section.

\begin{theorem}[Completeness]
    If $\Gamma\suppM{}{\,\gamma}\varphi^x$, then $\Gamma\vdash^{\gamma}\varphi^x$.
\end{theorem}
\begin{proof}
    Let $\flatIML{\cdot}$ be a flattening map with $\deflatIML{\cdot}$ being its corresponding left inverse and $\base{N}$ be a simulation base the sequent $(\Gamma:\varphi^x)$. We start by considering $\Gamma\suppM{}{\gamma}\varphi^x$, which, since it is valid, implies in particular that $\Gamma\suppIML{\base{N}}\varphi^x$ holds. By Lemma~\ref{lem:compFlattening}, this is equivalent to $\Gamma^\flat\suppIML{\base{N}}\flatIML{\varphi^x}$. By the (At) clause of the definition of $\suppIML{\base{N}}$ and Lemma~\ref{lem:basicInf}, we have that this is equivalent to $\flatIML{\Gamma}\deriveBaseIML{\base{N}}\flatIML{\varphi^x}$. Finally, by Lemma~\ref{lem:compNat}, this implies that $\deflatIML{\Gamma^\flat}\vdash^{\!\!\gamma}\deflatIML{\flatIML{\varphi^x}}$, which is $\Gamma\vdash^{\!\!\gamma}\varphi^x$, as desired. \fillBox
\end{proof}

\begin{proposition}\label{prop:compFlatLemma}
The following hold for all $\baseB\baseGeq\base{N}$:
    \begin{enumerate}
        \item[-]$\deriveBaseIML{\baseB}\flatIML{(\varphi\land\psi)^x}$ iff $\deriveBaseIML{\baseB}\flatIML{\varphi^x}$ and $\deriveBaseIML{\baseB}\flatIML{\psi^x}$
        \item[-]$\deriveBaseIML{\baseB}\flatIML{(\varphi\lor\psi)^x}$ iff for all $\baseC\baseGeq\baseB$ and atoms $p^z$ if $\varphi^x\suppIML{\baseC}p^z$ and $\psi^x\suppIML{\baseC}p^z$, then $\suppIML{\baseC}p^z$.
        \item[-]$\deriveBaseIML{\baseB}\flatIML{\bot^x}$ iff $\deriveBaseIML{\baseB}p^z$, for all atoms $p^z$.
        \item[-]$\deriveBaseIML{\baseB}\flatIML{(\Box\varphi)^x}$ iff for all labels $y$, we have $\flatIML{xRy} \deriveBaseIML{\baseB}\flatIML{\varphi^y}$.
        \item[-]$\deriveBaseIML{\baseB}\flatIML{(\Diamond\varphi)^x}$ iff for all $\baseC\baseGeq\baseB$ and atoms $p^z$, if $\flatIML{xRy},\flatIML{\varphi^y}\deriveBaseIML{\baseC}p^z$ for all $y$, then $\deriveBaseIML{\baseC}p^z$.
    \end{enumerate}
\end{proposition}
\begin{proof}
We take each case in turn:
    \begin{enumerate}
        \item[-] Left to right. From the hypothesis $\deriveBaseIML{\baseB}\flatIML{(\varphi\land\psi)^x}$, we have, by (App) applied to the rules $(\ern{\land_1}^\flat)$ and $(\ern{\land_2}^\flat)$, that $\deriveBaseIML{\baseB}\flatIML{\varphi^x}$ and $\deriveBaseIML{\baseB}\flatIML{\psi^x}$. 

        \medskip 
        
        Right to left. From the hypotheses $\deriveBaseIML{\baseB}\flatIML{\varphi^x}$ and $\deriveBaseIML{\baseB}\flatIML{\psi^x}$, we have again, by (App) applied to the rule $(\irn\land^\flat)$, that $\deriveBaseIML{\baseB}\flatIML{(\varphi\land\psi)^x}$.\\
        \item[-] Left to right. We obtain the conclusion immediately by (App) applied to the $(\ern \lor^\flat)$ rule. 

        \medskip 
        
        Right to left. Consider the bases $\baseC=\baseB$, and the case when $p^z = \flatIML{(\varphi\lor\psi)^x}$. In this case, we have sufficient grounds, by (App) applied to either the $(\irn{\lor_1}^\flat)$ rule or the $(\irn{\lor_2}^\flat)$ rule, to obtain $\deriveBaseIML{\baseB}\flatIML{(\varphi\lor\psi)^x}$, as required. \\
        \item[-] This case is immediate. \\
        \item[-] Left to right. We must show, given any label $z$, that $\flatIML{xRz} \deriveBaseIML{\baseB}\flatIML{\varphi^z}$ holds. To this end, we consider all bases $\baseC\baseGeq\baseB$, where $\deriveBaseIML{\baseC}\flatIML{xRz}$, and show that $\deriveBaseIML{\baseC}\flatIML{\varphi^z}$. Since we have $\deriveBaseIML{\baseB}\flatIML{(\Box\varphi)^x}$, by monotonicity, we have that $\deriveBaseIML{\baseC}\flatIML{(\Box\varphi)^x}$. Therefore, by (App) applied to the $(\ern \Box^\flat)$ rule, we obtain $\deriveBaseIML{\baseB}\flatIML{\varphi^z}$. 
        
        \medskip
        
        Right to left. We have that, for all labels $y$, $\flatIML{xRy} \deriveBaseIML{\baseB}\flatIML{\varphi^y}$. Consider a label $y$ that is different to $x$. Then, by (App) applied to the $(\irn \Box^\flat)$ rule, we obtain $\deriveBaseIML{\baseB}\flatIML{(\Box\varphi)^x}$.\\
        \item[-] Left to right. Start by fixing an arbitrary base $\baseC\baseGeq\baseB$ and atom $p^z$ such that $\flatIML{xRy},\flatIML{\varphi^y}\deriveBaseIML{\baseB}p^z$ for all $y$. Our goal will be to show that $\deriveBaseIML{\base{N}}p^z$. To this end, we note that, as we have $\flatIML{xRy},\flatIML{\varphi^y}\deriveBaseIML{\baseB}p^z$ for all $y$, we can pick some $y=a$ where $a\neq x,z$ to obtain $\flatIML{xRa},\flatIML{\varphi^a}\deriveBaseIML{\baseB}p^z$. Furthermore, by monotonicity, we also have that $\deriveBaseIML{\baseC}\flatIML{(\Diamond\varphi)^x}$. 
        Therefore, by (App) applied to the appropriate $(\ern \Diamond^\flat)$ rule which concludes $p^z$ (since there is one for each $y$ not equal to $x$ or $z$), we obtain $\deriveBaseIML{\baseC}p^z$, as required. 

        \medskip
        
        Right to left. The hypothesis gives that for all $\baseC\baseGeq\baseB$ and atoms $p^z$, if $\flatIML{xRy},\flatIML{\varphi^y}\deriveBaseIML{\baseC}p^z$ for all $y$, then $\deriveBaseIML{\baseC}p^z$. We start by considering $\baseC=\baseB$ and $p^z=\flatIML{(\Diamond\varphi)^x}$. Given an arbitrary variable $a$, we know that $\flatIML{xRa}, \flatIML{\varphi^a}\deriveBaseIML{\baseB}\flatIML{\varphi^a}$ and $\flatIML{xRa}, \flatIML{\varphi^a}\deriveBaseIML{\baseB}\flatIML{xRa}$ both hold by (Ref). 
        Therefore, by (App) applied to the appropriate $(\irn \Diamond^\flat)$ rule, we have $\flatIML{xRa},\flatIML{\varphi^a}\deriveBaseIML{\baseB}\flatIML{(\Diamond\varphi)^x}$ for arbitrary $a$. Thus, $\flatIML{xRy},\flatIML{\varphi^y}\deriveBaseIML{\baseB}\flatIML{(\Diamond\varphi)^x}$ holds for all $y$, from which we conclude, by the assumed implication, that $\deriveBaseIML{\baseB}\flatIML{(\Diamond\varphi)^x}$. \fillBox 
    \end{enumerate}
\end{proof}

\begin{lemma}\label{lem:compFlattening}
    $\suppIML{\baseB}\varphi^x$ if and only if $\suppIML{\baseB}\flatIML{\varphi^x}$.
\end{lemma}

\begin{proof}
We proceed by induction on the structure of $\varphi^x$. We give just one case as all other cases follow similarly. Consider the case in which $\varphi=\alpha\land\beta$. We have, by the definition of $(\land)$, that $\suppIML{\baseB}(\alpha\land\beta)^x$ if and only if $\suppIML{\baseB}\alpha^x$ and $\suppIML{\baseB}\beta^x$. By the inductive hypothesis, we have that this holds if and only if $\deriveBaseIML{\baseB}\flatIML{\alpha^x}$ and $\deriveBaseIML{\baseB}\flatIML{\beta^x}$. Thus, by Proposition~\ref{prop:compFlatLemma}, we know holds if and only if $\deriveBaseIML{\baseB}\flatIML{(\alpha\land\beta)^x}$, which by (At) holds if and only if $\suppIML{\baseB}\flatIML{(\alpha\land\beta)^x}$, as required. \fillBox
\end{proof}

\begin{lemma}\label{lem:compNat}
    If $L\deriveBaseIML{\base{N}}p^x$ then $L^\natural\vdash^{\!\!\gamma}\deflatIML{p^x}$, for any $L\subseteq\At$ and $p\in\At$.
\end{lemma}
\begin{proof}
    We consider how $L\deriveBaseIML{\base{N}}p^x$ obtains. If it obtains by (Ref), then it is the case that $p^x\in L$, and so we immediately have that $L^\natural\vdash^{\!\!\gamma}\deflatIML{p^x}$. Else, it must obtain by (App) or a modal case. In the case of (App), we argue by induction on the structure of the derivation which concluded it, taking the previous case as the base case.
    We thus have, by the inductive hypothesis, that if $((P_1\Rightarrow p_1,\dots, P_n\Rightarrow p_n)\Rightarrow r) \in \base{N}$ and for each $i$ that  $S^\natural,P_i^\natural\vdash^{\!\!\gamma}p_i^\natural$, then $S^\natural\vdash^{\!\!\gamma}r^\natural$. By letting the rule $((P_1\Rightarrow p_1,\dots, P_n\Rightarrow p_n)\Rightarrow r)$ range over all the rules of $\base{N}$ and noting that $\deflatIML{\flatIML{\varphi^x}}=\varphi^x$, whenever $\flatIML{\varphi^x}$ is defined, we see that we indeed to recover the rules for $N_{\Box\Diamond}(\gamma)$ and thus have the required deduction. In the modal cases, we observe that we always have rules in the base which, in effect, make the use of the modal cases of the derivability relation redundant. That is to say, in $\base{N}$ we can always replace an instance of an application of a ``modal case" in a derivation, with an instance of an application of (App) using a rule in $\base{N}$, at which point, we argue as in the (App) case. However, we must be careful in the modal cases $\rm(D)$ and $\rm(2)$. In these two cases, the fact that we have by hypothesis that there exists a $y$ ($w$ respectfully), allows us to instantiate the existential with a variable not appearing in the anywhere in the hypothesis. This, combined with the way in which we have constructed $\base{N}$, guarantees the existence of an appropriate rule in $\base{N}$ to allow us to argue as in the other modal cases. \fillBox
\end{proof}

A question to ask is whether in the cases of $\irn \Box$ and $\ern \Diamond$, the side-conditions are actually being adhered to? This is indeed the case because of the way in which $\base{N}$ is defined. 
Recall that the schemas $\irn \Box$ and $\ern\Diamond$ are simulated by ensuring that the base $\base{N}$ contains an instance of the rule for every possible allowed combination of labels. That this always guarantees the existence of the rule with the right kind of label is because of the way we add rules to the base. Recall that if a sequent contains a modal formula, say $(\Box p)^x$, then the set of generalized subformulae $\Xi$ will contain not only $(\Box p)^x$, but also $p^y$ and $xRy$, for all $y\in\Worlds$. Thus, the only rules in $\base{N}$ that allow introduction of $\flatIML{(\Box p)^x}$ are those that satisfy the side-conditions. Similar reasoning holds for $\ern\Diamond$. Therefore, when translating our proofs in $\base{N}$ under $\deflatILL{(\cdot)}$, we always obtain a valid $\rm N_{\Box\Diamond}(\gamma)$ proof.

\section{Conclusion} \label{sec:conclusion}

We have provided an inferentialist interpretation of a core family of intuitionistic modal logics, as defined by Simpson \cite{simpson1994Thesis}, through a base-extension semantics that uniformly and conservatively extends Sandqvist's 
base-extension semantics for IPL. Soundness and completeness properties have been established. This work makes essential use of the idea of labelling in proof systems that is familiar in work on tableaux systems and sequent calculi. 

The approach here stands in contrast to the work of Eckhardt and Pym on proof-theoretic semantics for classical modal logics \cite{EP1,EP2} in which `modal relations' are imposed on bases. 
Further work to understand the relationships between the two approaches would be valuable.


Base-extension semantics has also been developed for substructural logics, including intuitionistic multiplicative linear logic \cite{AlexTaoDavid_PtS4IMLL,GGP2024IMLL}, 
linear logic with additives and exponentials \cite{buzoku2025ILL}, and the bunched logic BI \cite{AlexTaoDavid_PtS4BI}.  

These logics typically come along with a `resource interpretation' or `resource semantics'. For example, in linear logic, propositions occurring in proofs can be interpreted as resources that can be consumed in the construction described by proof. Alternatively, in the logic of bunched implications (BI) \cite{GPSemBI2023,GGP2024}, a resource semantics resides in its Kripke-style models, with the intuitionistic connectives characterizing the sharing of resources and the multiplicative connectives characterizing the separation of resources. BI's resource semantics gives rise to Separation Logic, which is significant in program analysis and verification. 

Gheorghiu, Gu, and Pym, \cite{GGP2024}, have shown that BI's base-extension semantics provides a unifying framework for the resource readings described above. Furthermore, they show how base-extension semantics can be deployed as a system-modelling technology. Such a technology would be very much enriched by the inclusion of modal operators of the kind treated here and in \cite{EP1,EP2}. The use of base-extension semantics in system modelling appears to be a promising 
line of enquiry. This suggests that it would be worthwhile to explore 
automated reasoning tools for base-extension semantics for modal and substructural logics, though such a programme is beyond the scope of this paper.

\bibliographystyle{splncs04}
\bibliography{biblio}

\end{document}